\documentclass[reqno]{amsart}
\usepackage{latexsym,epsfig,amssymb,amsmath,amsthm,color,url}
\usepackage[comma,sort&compress]{natbib}
\usepackage{hyperref}

\allowdisplaybreaks 
 \numberwithin{equation}{section}
\bibfont{\footnotesize}
\newtheorem{thm}{Theorem}[section]
\newtheorem{lem}[thm]{Lemma}

\theoremstyle{definition}
\newtheorem{defn}[thm]{Definition}

\theoremstyle{remark}
\newtheorem{rem}[thm]{Remark}

\title{Inverse Problems under Sarmanov dependence structure}
\date{}
\author{Krishanu Maulik}
\address{Krishanu Maulik, \ Indian Statistical Institute, \ No. 203, Barrackpore Trunk Road, Kolkata 700108 \\ \tt{Email address : krishanu@isical.ac.in}.}
\author{Moumanti Podder}
\address{Moumanti Podder, \ Courant Institute of Mathematical Sciences, \ New York University, \ 251 Mercer Street, New York, NY 10012 \\ \tt{Email address : mp3460@nyu.edu}.}

\begin{document}
\bibliographystyle{plainnat}

\begin{abstract}
Consider a sequence $\{(X_{i}, Y_{i})\}$ of independent and identically distributed random vectors, with joint distribution bivariate Sarmanov. This is a natural set-up for discrete time financial risk models with insurance risks. Of particular interest are the infinite time ruin probabilities $P\left[\sup_{n \geq 1}\sum_{i=1}^{n} X_i \prod_{j=1}^{i}Y_{j} > x\right]$. When the $Y_{i}$'s are assumed to have lighter tails than the $X_{i}$'s, we investigate sufficient conditions that ensure each $X_{i}$ has a regularly varying tail, given that the ruin probability is regularly varying. This is an inverse problem to the more traditional analysis of the ruin probabilities based on the tails of the $X_{i}$'s. We impose moment-conditions as well as non-vanishing Mellin transform assumptions on the $Y_{i}$'s in order to achieve the desired results. But our analysis departs from the more conventional assumption of independence between the sequences $\{X_{i}\}$ and $\{Y_{i}\}$, instead assuming each $(X_{i}, Y_{i})$ to be jointly distributed as bivariate Sarmanov, a fairly broad class of bivariate distributions. 
\end{abstract}



\maketitle

\section{Introduction}
In this article we consider a discrete-time risk model with insurance and financial risks. We refer the reader to \cite{paulsen:2008} which describes the history and evolution of this model in detail. The survey discusses the pertinent integro-differential equations, asymptotics results and bounds on the ruin probability. It encompasses both continuous and discrete time risk model theories. The model we are concerned with constitutes the insurance risk $X_{n}$ and the financial risk or the stochastic discount factor $Y_{n}$ at time $n$. The stochastic discount value of aggregate net losses up to time $n$ in this set-up is given by 
\begin{equation} \label{discount value at n}
S_{n} = \sum_{i=1}^{n} X_{i} \prod_{j=1}^{i} Y_{j}.
\end{equation}
In this set-up the finite time ruin probabilities are defined as
\begin{equation} \label{finite time ruin}
\Psi(x, n) = P\left[\max_{1 \leq k \leq n} S_{k} > x\right],
\end{equation}
and the infinite time ruin probability is given by
\begin{equation} \label{infinite time ruin}
\Psi(x) = P\left[\sup_{n \geq 1} S_{n} > x\right].
\end{equation}
In general we derive the behaviour of $\Psi(x, n)$ and $\Psi(x)$ from the tail of $X_{1}$. However, the focus of this paper will be to study the inverse problem, i.e. the tail of $X_{1}$ given the behaviour of $\Psi(x, n)$ and $\Psi(x)$.

\par In the risk-model set-up described above, each $X_{i}$ is generally assumed to follow a regularly varying distribution. \cite{resnick:1987} has studied this class of distributions extensively, and its uses in applied probability are detailed in \cite{goldie:1987}. Its applications in stochastic recurrence equations have been studied in \cite{basrak:2002} and \cite{denisov:zwart:2007} among others. \cite{nyrhinen:2012} and \cite{yang:wang:2013} have investigated the myriad applications of regularly varying distributions in the above-mentioned risk-model problems. In particular, \cite{yang:wang:2013} considers a set up where the random vectors $(X_{n}, Y_{n})$ are i.i.d.\ jointly distributed as bivariate Sarmanov, as in Definition \ref{sarmanov definition}.

\par Recall that a function $f$ is said to be regularly varying at $\infty$ with index $\beta$ if $f(xy) \sim y^{\beta}f(x)$ as $x \rightarrow \infty$. Similarly, a random variable $X$ is said to have a regularly varying tail of index $-\alpha$, $\alpha > 0$, denoted by $X \in RV_{-\alpha}$, if its tail distribution $\overline{F}$ satisfies, for all $y > 0$, 
$$\overline{F}(xy) \sim y^{-\alpha} \overline{F}(x), \quad \text{as } x \rightarrow \infty.$$
Here and henceforth, for two positive functions $a(x)$ and $b(x)$, we write $a(x) \sim b(x)$ as $x \rightarrow \infty$ to mean that $\lim_{x \rightarrow \infty} a(x)/b(x) = 1$. When $X \in RV_{-\alpha}$ and $Y$ is independent of $X$, satisfying $E[Y^{\alpha + \epsilon}] < \infty$ for some $\epsilon >0$,  \cite{breiman:1965} shows that $XY$ will also be regularly varying with $P[XY > x] \sim E[Y^{\alpha}] P[X > x]$ as $x \rightarrow \infty$. The \emph{inverse problem} corresponding to this same set-up has been studied by \cite{jacobsen:2009}. When $XY$ is regularly varying, they propose sufficient conditions for $X$ to be regularly varying. \cite{resnick:willekens:1991} extended the product result of Breiman to finite and infinite sums. \cite{hazra:maulik:2012} explored the inverse problems corresonding to the finite and infinite sums under the Resnick-Willekens conditions. In Section \ref{sec:product} we are interested in the inverse problem for products, but with the classically studied independence structure between $X$ and $Y$ replaced by the bivariate Sarmanov distribution as in Definition \ref{sarmanov definition}. Yang and Wang obtained sufficient conditions for regular variation of $\Psi(x, n)$ and $\Psi(x)$ when each $X_{n}$ is regularly varying. \cite{maulik:podder:2016} extended their results for conditions similar to those proposed by \cite{denisov:zwart:2007}. In Sections \ref{sec:finite sum} and \ref{sec:infinite sum} we shall be interested in analyzing the inverse problem for sums -- finite and infinite cases respectively, imposing Breiman-like moment conditions on the $Y_{i}$'s. In section \ref{non-vanishing Mellin} we show the necessity of non-vanishing Mellin transform of the appropriate measure. We are motivated by the example given in \cite{jacobsen:2009} which we adapt appropriately for our set-up. In our study, each of the i.i.d.\ random vectors $(X_{i}, Y_{i})$ follows a bivariate Sarmanov distribution as in Definition \ref{sarmanov definition}. Given that $\Psi(x, n)$ or $\Psi(x) \in RV_{-\alpha}$, we shall investigate sufficient conditions to ensure that $X_{n} \in RV_{-\alpha}$ for each $n$, under the above mentioned dependence assumptions. In this context, we also refer to \cite{damek:2014} for discussions on inverse problems for regular variations in multivariate cases and where regular variation is not restricted to one direction or quadrant. They discuss inverse problems for the convolution of two multivariate random measures, assuming independence between them. They also focus on the inverse problems for weighted sums of multivariate regularly varying measures, but with the weights being non-random matrices. Our work in this paper departs both from their independence assumption as well as deterministic weights.

\par We give a brief description of the results from \cite{jacobsen:2009}, as our results are vitally based on these. Given a probability measure $\nu$ and a $\sigma$-finite measure $\rho$ on $(0, \infty)$, we define a new measure on $(0, \infty)$ by
$$(\nu \circledast \rho)(B) = \int_{0}^{\infty} \nu(x^{-1}B)\rho(dx), \quad B \text{ a Borel set on } (0, \infty).$$
Following Jacobsen et. al., we call this the \emph{multiplicative convolution} of $\nu$ and $\rho$, since in the case where $\nu$ and $\rho$ are probability measures, $\nu \circledast \rho$ gives the law of the product of two independent random variables with marginals $\nu$ and $\rho$. 

%A $\sigma$-finite measure $\rho$ is said to have the \emph{cancellation property} with respect to a family $\mathcal{N}$ of $\sigma$-finite measures on 
%$(0, \infty)$ if for any $\nu$ $\sigma$-finite on $(0, \infty)$ and $\overline{\nu} \in \mathcal{N}$, 
%$$\nu \circledast \rho = \overline{\nu} \circledast \rho \quad \implies \quad \nu = \overline{\nu}.$$ Restricting ourselves to what Jacobsen et. al. 
%considered, we will only deal with $\mathcal{N} = \{\nu_{\alpha}\}$ where $\alpha > 0$ and $\nu_{\alpha} \in RV_{-\alpha}$.

We now provide a paraphrased version of Theorem 2.3 of \cite{jacobsen:2009} which inspires our main result in Section \ref{sec:product}. 
\begin{thm} \label{jacobsen}
Let $\rho$ be a non-zero $\sigma$-finite measure such that, for some $\alpha > 0$,
\begin{equation} \label{moment condition}
\int_{0}^{\infty} y^{\alpha - \delta} \vee y^{\alpha + \delta}\rho(dy) < \infty, \quad \text{for some } \delta > 0.
\end{equation}
and the non-vanishing Mellin transform condition holds, i.e.
\begin{equation} \label{non-vanishing Mellin transform}
\int_{0}^{\infty} y^{\alpha + i \theta} \rho(dy) \neq 0 \quad \text{for all } \theta \in \mathbb{R}.
\end{equation}
%Suppose for some probability measure $\nu$ on $(0, \infty)$, the measure $\nu \circledast \rho \in RV_{-\alpha}$, and
%\begin{equation} \label{unif integ}
%\lim_{b \rightarrow 0} \limsup_{x \rightarrow \infty} \frac{\int_{0}^{\infty} \rho(x/y, \infty) \nu(dy)}{(\nu \circledast \rho)(x, \infty)} = 0.
%\end{equation}
If $\nu \circledast \rho \in RV_{-\alpha}$, then $\nu \in RV_{-\alpha}$ as well, with 
$$\lim_{x \rightarrow \infty} \frac{(\nu \circledast \rho)(x, \infty)}{\nu(x, \infty)} = \int_{0}^{\infty} y^{\alpha} \rho(dy).$$
\end{thm}

\begin{rem}
In their original result, Jacobsen et al.\ allowed $\nu$ to be a $\sigma$-finite measure, for which they required an additional integrability assumption. It is not needed when $\nu$ is a probability measure.
\end{rem}

%Furthermore, if the left side of \eqref{non-vanishing Mellin transform} vanishes for some $\theta_{0} \in \mathbb{R}$, then for any real $a, b$ with $0 < %a^{2} + b^{2} \leq 1$, the $\sigma$-finite measure 
%\begin{equation} \label{counterexample}
%\nu(dx) = g(x) \nu_{\alpha}(dx)
%\end{equation}
%with 
%\begin{equation} \label{g}
%g(x) = 1 + a \cos (\theta_{0} \log x) + b \sin (\theta_{0} \log x), \quad x > 0,
%\end{equation}
%satisfies $\nu \circledast \rho = \nu_{\alpha} \circledast \rho$.
%\end{equation}
\par We relax the independence between $\nu$ and $\rho$ as considered by Jacobsen et al.\ in defining the product convolution $\circledast$. We extend our dependence structure to the much broader class of bivariate Sarmanov distributions, which is defined as follows.
\begin{defn} \label{sarmanov definition}
The pair of random variables $(X,Y)$ is said to follow a bivariate Sarmanov distribution, if
$$P(X \in dx, Y \in dy) = (1+\theta \phi_{1}(x)\phi_{2}(y))F(dx)G(dy), \quad x \in \mathbb{R}, y \geq 0,$$
where the kernels $\phi_{1}$ and $\phi_{2}$ are two real valued functions and the parameter $\theta$ is a real constant satisfying 
$$E\{\phi_{1}(X)\} = E\{\phi_{2}(Y)\} = 0$$
and 
$$1+\theta \phi_{1}(x)\phi_{2}(y) \geq 0, \quad x \in D_{X}, y \in D_{Y},$$ where $D_{X} \subset \mathbb{R}$ and $D_{Y} \subset \mathbb{R}^{+}$ are the supports of $X$ and $Y$, with marginals $F$ and $G$ respectively.
\end{defn}
\par This class of bivariate distributions is quite wide, covering a large number of well-known copulas such as the Farlie-Gumbel-Morgenstern (FGM) copula, which is recovered by taking $\phi_{1}(x) = 1 - 2F(x)$ and $\phi_{2}(y) = 1 - 2G(y)$. We refer the reader to \cite{lee:1996} for further discussion. A bivariate Sarmanov distribution is called proper if $\theta \neq 0$ and none of $\phi_{1}$ and $\phi_{2}$ vanishes identically.

\par As has been discussed above, Yang and Wang studied this class of distributions. They additionally assumed
\begin{equation} \label{limit of phi}
\lim_{x \rightarrow \infty} \phi_{1}(x) = d_{1}.
\end{equation}
Yang and Wang made the crucial observation that the bivariate Sarmanov dependence is not very far from independence. If $(X, Y)$ is bivariate Sarmanov, then asymptotically, the product $XY$ has the same tail distribution as the product $X Y^{*}_{\theta}$ where $X$ and $Y^{*}_{\theta}$ are independent and $Y^{*}_{\theta}$ is obtained through a change of measure performed on the distribution of $Y$. It has the distribution function $G_{\theta}$ with \begin{equation} \label{twisted version}
G_{\theta}(dy) = P[Y^{*}_{\theta} \in dy] = (1 + \theta d_{1} \phi_{2}(y))G(dy).
\end{equation}

To state the result formally, we first need to define the class of dominatedly-tail-varying distributions. A random variable $X$ with distribution function $F$ is called dominatedly-tail-varying, denoted by $X \in \mathcal{D}$ or $F \in \mathcal{D}$, if for all $0 < y < 1$, 
\begin{equation} \label{dominated tail}
\limsup_{x \rightarrow \infty}\frac{\overline{F}(xy)}{\overline{F}(x)} < \infty.
\end{equation} 

Lemma 3.1 of \cite{yang:wang:2013} shows the weak dependence of the bivariate Sarmanov distribution, but we shall need a less generalized version of it, stated as follows.
\begin{thm} \label{almost:independent}
Assume that $(X,Y)$ follows a bivariate Sarmanov distribution and \eqref{limit of phi} holds. Let $X^{*}$ and $Y^{*}$ be two independent random variables identically distributed as $X$ and $Y$ respectively, i.e. having marginals $F$ and $G$ respectively. Let $\overline{H^{*}}(x) = P[X^{*}Y^{*} > x]$. If now $H^{*} \in \mathcal{D}$ and $\overline{G}(x) = o(\overline{H^{*}}(x)),$ then 
\begin{equation} \label{eq: sarmanov}
P[X Y > x] \sim P[X^{*}Y^{*}_{\theta} > x],
\end{equation}
where $X^{*}, Y_{\theta}^{*}$ mutually independent and $Y_{\theta}^{*} \sim G_{\theta}$ as defined in \eqref{twisted version}.
\end{thm}

Yang and Wang also noted the following property of bivariate Sarmanov which will also be important for establishing our results.
\begin{lem} \label{kernel:bounded}
Assume that $(X,Y)$ follows a proper bivariate Sarmanov distribution. Then there exists two positive constants $b_{1}$ and $b_{2}$ such that $|\phi_{1}(x)| \leq b_{1}$ for all $x \in D_{X}$ and $|\phi_{2}(y)| \leq b_{2}$ for all $y \in D_{Y}$.
\end{lem}

\section{Inverse problem for product} \label{sec:product}
We now state our result concerning the tail of one of the multiplicands from the regularly varying tail of the product of two random variables. 
\begin{thm} \label{inverse product}
Suppose the pair of random variables $(X, Y)$ jointly follow bivariate Sarmanov distribution, as defined in Definition \ref{sarmanov definition}, with $\lim_{x \rightarrow \infty} \phi_{1}(x) = d_{1}$. We also assume that $F \in \mathcal{D}$ and $\overline{G}(x) = o(\overline{F}(x))$. Suppose $XY \in RV_{-\alpha}$ for some $\alpha > 0$.  If now we have $E[Y^{\alpha + \epsilon}] < \infty$ for some $\epsilon > 0$ and for all $\beta \in \mathbb{R}$, 
\begin{equation} \label{non-vanishing Mellin twisted}
E[Y^{\alpha + i \beta}] + \theta d_{1} E[\phi_{2}(Y) Y^{\alpha + i \beta}] \neq 0,
\end{equation}
then $X \in RV_{-\alpha}$ and $P[XY > x] \sim \left\{E[Y^{\alpha}] + \theta d_{1} E[\phi_{2}(Y) Y^{\alpha}]\right\} P[X > x]$.
\end{thm}

\par We shall assume without loss of generality that $\epsilon \in (0, \alpha)$. 

%\begin{rem}
%To prove Theorem \ref{inverse product}, as well as in Sections \ref{sec:finite sum} and \ref{sec:infinite sum}, we shall require one of the conclusions of %Theorem 3.3 of \cite{cline:1994}. Suppose $F$ and $G$ are two probability distributions, both not degenerate at $0$. Let $X$ and $Y$ be mutually %independent with $X \sim F, \ Y \sim G$. Then
%\begin{equation} \label{dominated tail of product}
%X \in \mathcal{D} \quad \implies \quad XY \in \mathcal{D}.
%\end{equation}
%\end{rem}

\begin{proof} [Proof of Theorem \ref{inverse product}]
Let $X^{*}, Y^{*}$ be mutually independent copies of $X$ and $Y$, with marginals $F$ and $G$ respectively. We define $H^{*}$ by $\overline{H^{*}}(x) = P[X^{*} Y^{*} > x]$. Since $F \in \mathcal{D}$, from Theorem 3.3 of \cite{cline:1994}, we conclude that $H^{*} \in \mathcal{D}$ as well.
\par Choosing a suitable $a$ such that $\overline{G}(a) > 0$, we have
$$\overline{H^{*}}(x) \geq \overline{F}(x/a) \overline{G}(a).$$
Therefore, using $F \in \mathcal{D}$ and $\overline{G}(x) = o(\overline{F}(x))$, we have
$$\limsup_{x \rightarrow \infty} \frac{\overline{G}(x)}{\overline{H^{*}}(x)} \leq \limsup_{x \rightarrow \infty} \frac{\overline{G}(x)}{\overline{F}(x)} \limsup_{x \rightarrow \infty} \frac{\overline{F}(x)}{\overline{F}(x/a)} \limsup_{x \rightarrow \infty} \frac{\overline{F}(x/a)}{\overline{H^{*}}(x)} = 0.$$
Thus we have established that $\overline{G}(x) = o(\overline{H^{*}}(x))$.
%\begin{equation} \label{conditions for almost:independent}
%H^{*} \in \mathcal{D} \quad \text{and} \quad \overline{G}(x) = o(\overline{H^{*}}(x)).
%\end{equation}
Recall the twisted version $Y_{\theta}^{*}$ of $Y$ defined in \eqref{twisted version}. Then by Theorem \ref{almost:independent}, we know that $$P[X^{*}Y_{\theta}^{*} > x] \sim P[XY > x] \quad \implies \quad X^{*}Y_{\theta}^{*} \in RV_{-\alpha}.$$ 
Using Lemma \ref{kernel:bounded}, we get
$$E[{Y_{\theta}^{*}}^{\alpha + \epsilon}] \leq (1 + |\theta d_{1}| b_{2}) E[Y^{\alpha + \epsilon}] < \infty.$$
As defined in \eqref{twisted version}, if $G_{\theta}$ denotes the marginal of $Y_{\theta}^{*}$, then
\begin{align} 
\int_{0}^{\infty} y^{\alpha - \epsilon} \vee y^{\alpha + \epsilon} G_{\theta}(dy) =&  \int_{0}^{1} y^{\alpha - \epsilon} G_{\theta}(dy) + \int_{1}^{\infty} y^{\alpha + \epsilon} G_{\theta}(dy) \leq 1 + E[{Y_{\theta}^{*}}^{\alpha + \epsilon}] < \infty. \nonumber
\end{align}
By \eqref{non-vanishing Mellin twisted}, for all $\beta \in \mathbb{R}$, we have $\int_{0}^{\infty} y^{\alpha + i \beta} G_{\theta}(dy) \neq 0.$ We are now able to conclude, from Theorem \ref{jacobsen}, that $X^{*}$ and hence $X$ is in $RV_{-\alpha}$. The final result follows using \cite{breiman:1965}'s theorem.
%Using \cite{breiman:1965}'s theorem,
%$P[X^{*}Y_{\theta}^{*} > x] \sim E[{Y_{\theta}^{*}}^{\alpha}] P[X^{*} > x]$ which implies, using Theorem \ref{almost:independent},
%$$P[XY > x] \sim \left\{E[Y^{\alpha}] + \theta d_{1} E[\phi_{2}(Y) Y^{\alpha}]\right\} P[X > x].$$

\end{proof}

%\begin{rem}
%If we assume that $H^{*} \in \mathcal{D}$ and $\overline{G}(x) = o(\overline{H^{*}}(x))$, then the proof will go through without the need to invoke the %result from \cite{cline:1994}.
%\end{rem}

\section{Inverse problem for finite sum} \label{sec:finite sum}
We start with the same set-up as described in \cite{yang:wang:2013}. Let $\{(X_{i}, Y_{i})\}$ be a sequence of i.i.d.\ random vectors, with the generic vector $(X, Y)$ jointly having bivariate Sarmanov distribution, as in Definition \ref{sarmanov definition}. Recall that $\Psi(x, n)$ is the finite time ruin probability defined as $\Psi(x, n) = P\left[\displaystyle \max_{1 \leq k \leq n} S_{k} > x \right]$ where $S_{n}$ is as in \eqref{discount value at n}. We provide sufficient conditions under which $\Psi(x, n) \in RV_{-\alpha}$ implies $X \in RV_{-\alpha}$. To this end, we state and prove the following important lemma.

\begin{lem} \label{lemma 1 finite sum}
Let $\{(X_{i}, Y_{i})\}$ be a sequence of i.i.d.\ random vectors, with the generic vector $(X, Y)$ jointly bivariate Sarmanov distribution, as in Definition \ref{sarmanov definition}. We assume that $F \in \mathcal{D}$ and $\overline{G}(x) = o(\overline{F}(x))$. Denoting $P[XY > x]$ by $\overline{H}(x)$, we assume that for every $v > 0$, the quantity
\begin{equation} \label{sup for H}
\widetilde{H}(v) = \sup_{x > 0}\frac{\overline{H}(x/v)}{\overline{H}(x)}
\end{equation}
satisfies the condition 
\begin{equation} \label{unif integ finite sum}
\int_{0}^{\infty} \widetilde{H}(v) G(dv) < \infty.
\end{equation}
Then we can conclude that
\begin{equation} \label{eq: lemma 1 finite sum}
P[X_{1}Y_{1} > x, \ X_{2} Y_{2} Y_{1} > x] = o(P[X_{1} Y_{1} > x]).
\end{equation}
\end{lem}

%\begin{rem}
%We wish to show, for all $v > 0$, that 
%\begin{equation} \label{sup finite}
%\sup_{x > 0} \frac{\overline{H}(x/v)}{\overline{H}(x)} < \infty.
%\end{equation}
%When $0 < v < 1$, $\frac{\overline{H}(x/v)}{\overline{H}(x)} \leq 1$. Now suppose $X^{*}, Y^_{\theta}^{*}$ are mutually independent with $X^{*} \sim 
%F$ and $Y_{\theta}^{*} \sim G_{\theta}$ as defined in \eqref{twisted version}. Because $F \in \mathcal{D}$, using \eqref{dominated tail of product} we 
%conclude that $X^{*}Y_{\theta}^{*} \in \mathcal{D}$ as well. From Theorem \ref{almost:independent} we know that $\overline{H}(x) = P[XY > x] \sim %P[X^{*}Y_{\theta}^{*} > %x]$. Hence $XY \in \mathcal{D}$, which implies
%$$\limsup_{x \rightarrow \infty}
%\end{rem}

\begin{proof}
Conditioning on $Y_{1}$, and noting that $(X_{2}, Y_{2})$ independent of $(X_{1}, Y_{1})$, we get
\begin{equation} \label{breakdown 1}
\frac{P[X_{1}Y_{1} > x, \ X_{2} Y_{2} Y_{1} > x]}{P[X_{1} Y_{1} > x]} = \frac{\int_{0}^{\infty} P[X_{1} > x/v|Y_{1} = v] \overline{H}(x/v) G(dv)}{\overline{H}(x)}.
\end{equation}
Using the bivariate dependence structure between $X_{1}$ and $Y_{1}$, and Lemma \ref{kernel:bounded}, we have
\begin{equation} \label{bound on breakdown 1}
P[X_{1} > x/v|Y_{1} = v] \leq \int_{x/v}^{\infty} (1 + |\theta| |\phi_{1}(u)| |\phi_{2}(v)|) F(du) \leq (1 + |\theta| b_{1} b_{2}) \overline{F}(x/v).
\end{equation}
Observe that $\widetilde{H}(v) - \frac{\overline{F}(x/v) \overline{H}(x/v)}{\overline{H}(x)} \geq 0$, where $\widetilde{H}$ is as in \eqref{sup for H}. Applying Fatou's lemma, we now get
\begin{align} \label{after fatou}
\int_{0}^{\infty} \widetilde{H}(v) G(dv) - \limsup_{x \rightarrow \infty}\int_{0}^{\infty}\frac{\overline{F}(x/v) \overline{H}(x/v)}{\overline{H}(x)}G(dv) \geq & \int_{0}^{\infty} \left[\widetilde{H}(v) - \limsup_{x \rightarrow \infty} \frac{\overline{F}(x/v) \overline{H}(x/v)}{\overline{H}(x)}\right] G(dv) \nonumber\\
=& \int_{0}^{\infty} \widetilde{H}(v) G(dv),
\end{align}
where the last equality follows from the assumption of \eqref{unif integ finite sum}.
From \eqref{after fatou} we conclude that
\begin{equation} \label{lim sup 0}
\limsup_{x \rightarrow \infty}\int_{0}^{\infty}\frac{\overline{F}(x/v) \overline{H}(x/v)}{\overline{H}(x)}G(dv) = 0.
\end{equation}
Combining \eqref{breakdown 1}, \eqref{bound on breakdown 1} and \eqref{lim sup 0}, we get the desired result.
\end{proof}

The next lemma is stated under the same conditions as Lemma \ref{lemma 1 finite sum}. It additionally assumes Breiman-like moment condition on the random variables $Y_{i}$, and that the tail of the product $X_{1}Y_{1}$ is bounded above by a regularly varying function of index $-\alpha$. Under these additional assumptions, the next lemma shows negligibility of joinyt tails of higher products with respect to the general dominator, which is a regularly varying function.

\begin{lem} \label{lemma 2 finite sum}
Let $\{(X_{i}, Y_{i})\}$ be a sequence of i.i.d.\ random vectors, with the generic vector $(X, Y)$ jointly bivariate Sarmanov distribution, as in Definition \ref{sarmanov definition}, and satisfying $\lim_{x \rightarrow \infty}\phi_{1}(x) = d_{1}$. We assume that $F \in \mathcal{D}$ and $\overline{G}(x) = o(\overline{F}(x))$. Denoting $P[XY > x]$ by $\overline{H}(x)$, and $\widetilde{H}$ as in \eqref{sup for H}, we assume that the condition \eqref{unif integ finite sum} holds. Suppose $X$ is nonnegative and 
\begin{equation} \label{reg var bound}
P[XY > x] \leq W(x), \quad \text{for all } x > 0,
\end{equation}
where $W$ is a bounded regularly varying function with index $-\alpha$. Finally, we assume that $E[Y^{\alpha + \epsilon}] < \infty$ for some $\epsilon > 0$. Then, for $1 \leq s < t \leq n$, 
\begin{equation} \label{eq:lemma 2 finite sum}
P\left[X_{s}\prod_{i=1}^{s}Y_{i} > x, \ X_{t}\prod_{i=1}^{t}Y_{i} > x\right] = o(W(x)), \quad \text{as } x \rightarrow \infty.
\end{equation}

\end{lem} 

\begin{proof}
Let $Z_{s} = X_{s} Y_{s}, \ Z_{t} = X_{t} Y_{t} Y_{s}, \ \Theta_{s} = \prod_{j=1}^{s-1}Y_{j}$ and $\Theta_{t} = \prod_{1 \leq j \leq t-1, j \neq s}Y_{j}$. Observe that $(Z_{s}, Z_{t})$ is independent of $(\Theta_{s}, \Theta_{t})$. Let $G_{s, t}$ be the joint distribution of $(\Theta_{s}, \Theta_{t})$, and $G_{s}, G_{t}$ the respective marginals. We condition on $(\Theta_{s}, \Theta_{t})$ to get
\begin{align} \label{bound lemma 2 finite sum}
\frac{P\left[X_{s}\prod_{i=1}^{s}Y_{i} > x, \ X_{t}\prod_{i=1}^{t}Y_{i} > x\right]}{W(x)}  = & \int_{0}^{\infty} \int_{0}^{\infty} \frac{P[Z_{s} > x/u, Z_{t} > x/v]}{W(x)}G_{s, t}(du, dv) \nonumber\\
\leq & \int_{u > v} \frac{P[Z_{s} > x/u, Z_{t} > x/u]}{W(x)}G_{s,t}(du, dv) \nonumber\\ &+ \int_{u \leq v} \frac{P[Z_{s} > x/v, Z_{t} > x/v]}{W(x)}G_{s, t}(du, dv) \nonumber\\
\leq & \int_{0}^{\infty} \frac{P[Z_{s} > x/u, Z_{t} > x/u]}{W(x)} (G_{s} + G_{t})(du) \nonumber\\
\leq & \int_{0}^{\infty} \frac{P[X_{1}Y_{1} > x/u, X_{2}Y_{2}Y_{1} > x/u]}{P[X_{1}Y_{1} > x/u]} \cdot \frac{W(x/u)}{W(x)} (G_{s} + G_{t})(du).
\end{align}

\par We observe that since $W \in RV_{-\alpha}$, using Potter's bounds from \cite{resnick:1987} we can choose a suitable $x_{0} > 0$ and $M > 0$  such that for all $x > x_{0}$, we have

%\[\frac{W(x/u)}{W(x)} \leq \left\{\begin{array}{ll} 
%Mu^{\alpha - \varepsilon} & \text{if $u < 1$}\\
%Mu^{\alpha + \varepsilon} &\text{if $1 \leq u \leq x/x_{0}$}
%\end{array}\right.\]

\begin{equation}
  \frac{W(x/u)}{W(x)} \leq \begin{cases}
    Mu^{\alpha - \varepsilon} & \text{if } u < 1 \\
    Mu^{\alpha + \varepsilon} & \text{if } 1 \leq u \leq x/x_{0}. 
  \end{cases} \label{Potter}
\end{equation}
For $u \leq x/x_{0}$, we bound the integrand of \eqref{bound lemma 2 finite sum}  by $M(1 + u^{\alpha+\epsilon})$. 

%Observe that
%$$\int_{0}^{x/x_{0}} M(1 + u^{\alpha + \epsilon}) (G_{s} + G_{t})(du) \leq 2M + M\left\{E[\Theta_{s}^{\alpha + \epsilon}] + E[\Theta_{t}^{\alpha + %\epsilon}]\right\} < \infty.$$
%From Lemma \ref{lemma 1 finite sum} and as $W \in RV_{-\alpha}$, we have 
%$$\lim_{x \rightarrow \infty} \frac{P[X_{1}Y_{1} > x/u, X_{2}Y_{2}Y_{1} > x/u]}{P[X_{1}Y_{1} > x/u]} \cdot \frac{W(x/u)}{W(x)} = 0,$$
%so that, 

Using Lemma \ref{lemma 1 finite sum} and applying dominated convergence theorem, we get
$$\lim_{x \rightarrow \infty} \int_{0}^{x/x_{0}} \frac{P[X_{1}Y_{1} > x/u, X_{2}Y_{2}Y_{1} > x/u]}{P[X_{1}Y_{1} > x/u]} \cdot \frac{W(x/u)}{W(x)} (G_{s} + G_{t})(du) = 0.$$

\par For $u > x/x_{0}$, the integral of \eqref{bound lemma 2 finite sum} is bounded above by
\begin{align}
\sup_{y > 0} W(y) \cdot \frac{P[\Theta_{t} > x/x_{0}] + P[\Theta_{s} > x/x_{0}]}{W(x)} \leq & \sup_{y > 0} W(y) \cdot \frac{x_{0}^{\alpha + \epsilon} \left\{E[\Theta_{t}^{\alpha + \epsilon}] + E[\Theta_{s}^{\alpha + \epsilon}]\right\}}{x^{\alpha + \epsilon} W(x)}\nonumber
\end{align}
which goes to zero as $x \rightarrow \infty$.
 
\end{proof}

\par With the aid of Lemmas \ref{lemma 1 finite sum} and \ref{lemma 2 finite sum}, we are now able to state and prove the main result of this section.
\begin{thm} \label{inverse finite sum}
Let $\{(X_{i}, Y_{i})\}$ be a sequence of i.i.d.\ random vectors, with the generic vector $(X, Y)$ jointly bivariate Sarmanov distribution, as in Definition \ref{sarmanov definition}, and satisfying $\lim_{x \rightarrow \infty}\phi_{1}(x) = d_{1}$. We assume that $F \in \mathcal{D}$ and $\overline{G}(x) = o(\overline{F}(x))$. Denoting $P[XY > x]$ by $\overline{H}(x)$, and $\widetilde{H}$ as in \eqref{sup for H}, we assume that the condition of \eqref{unif integ finite sum} holds. Suppose $X$ is nonnegative and $S_{n} \in RV_{-\alpha}$, where $S_{n}$ is as defined in \eqref{discount value at n}. We assume that $E[Y^{\alpha + \epsilon}] < \infty$ for some $\epsilon > 0$. For all $\beta \in \mathbb{R}$, 
\begin{equation} \label{product Mellin}
E[Y^{\alpha + i \beta}] + \theta d_{1} E[\phi_{2}(Y) Y^{\alpha + i \beta}] \neq 0,
\end{equation}
and
\begin{equation} \label{finite sum Mellin}
\sum_{k=0}^{n-1} \left\{E[Y^{\alpha + i \beta}]\right\}^{k} \neq 0.
\end{equation}
Then each $X_{i} \in RV_{-\alpha}$ and 
\begin{equation} \label{eq: inverse finite sum}
P[S_{n} > x] \sim \frac{\left(1 - E[Y^{\alpha}]^{n}\right) \left(E[Y^{\alpha}] + \theta d_{1} E[\phi_{2}(Y) Y^{\alpha}]\right)}{\left(1 - E[Y^{\alpha}]\right)} \overline{F}(x), \quad \text{as } x \rightarrow \infty.
\end{equation}
\end{thm}

\begin{proof}
From Lemma 4.3 of \cite{hazra:maulik:2012}, we know that, for any $1/2 < \delta < 1$, 
\begin{equation} \label{lower bound}
P\left[S_{n} > x\right] \geq \sum_{i=1}^{n}P\left[X_{i}\prod_{j=1}^{i}Y_{j} > x\right] - \sum_{1 \leq s \neq t \leq n}P\left[X_{s}\prod_{j=1}^{s}Y_{j} > x, X_{t}\prod_{j=1}^{t}Y_{j} > x\right],
\end{equation}
and
\begin{equation} \label{upper bound}
P\left[S_{n} > x\right] \leq \sum_{i=1}^{n}P\left[X_{i}\prod_{j=1}^{i}Y_{j} > \delta x\right] + \sum_{1 \leq s \neq t \leq n}P\left[X_{s}\prod_{j=1}^{s}Y_{j} > \frac{1-\delta}{n-1}x, X_{t}\prod_{j=1}^{t}Y_{j} > \frac{1-\delta}{n-1}x\right].
\end{equation}
From \eqref{lower bound} and Lemma \ref{lemma 2 finite sum}, we have
\begin{equation} \label{lower bound limit}
\limsup_{x \rightarrow \infty} \frac{\sum_{i=1}^{n}P\left[X_{i}\prod_{j=1}^{i}Y_{j} > x\right]}{P[S_{n} > x]} \leq 1.
\end{equation}
From \eqref{upper bound} we have
\begin{multline} \label{liminf upper bound}
\liminf_{x \rightarrow \infty} \frac{\sum_{i=1}^{n}P\left[X_{i}\prod_{j=1}^{i}Y_{j} > x\right]}{P[S_{n} > x]} \geq \liminf_{x \rightarrow \infty} \frac{P\left[S_{n} > x\right]}{P\left[S_{n} > \delta x\right]} \\ - \sum_{1 \leq s \neq t \leq n} \limsup_{x \rightarrow \infty} \frac{P\left[X_{s}\prod_{j=1}^{s}Y_{j} > \frac{1-\delta}{n-1}x, \ X_{t}\prod_{j=1}^{t}Y_{j} > \frac{1-\delta}{n-1}x\right]}{P\left[S_{n} > \delta x\right]} 
\end{multline} 
and the right side of \eqref{liminf upper bound} equals $\delta^{\alpha}$, using the regular variation of $S_{n}$ and Lemma \ref{lemma 2 finite sum}. Hence letting $\delta \rightarrow 1$, we have
\begin{equation} \label{finite sum of tails tail of sum}
P[S_{n} > x] \sim \sum_{k=1}^{n} P[X_{k}\prod_{j=1}^{k}Y_{j} > x] \quad \text{as } x \rightarrow \infty.
\end{equation}
Let $\nu$ be the law induced by each $X_{i}Y_{i}$, and $\rho$ be the law given by
$$\rho(B) = \sum_{k=1}^{n} P\left[\prod_{j=1}^{k-1}Y_{j} \in B\right], \quad B \text{ a Borel set in } (0, \infty),$$
where the empty product is defined as $1$. Since $E[Y^{\alpha + \epsilon}] < \infty$, $\rho$ is a finite measure satisfying the moment condition \eqref{moment condition}. Due to \eqref{finite sum Mellin}, it also satisfies the non-vanishing Mellin transform condition \eqref{non-vanishing Mellin transform}. The multiplicative convolution $\nu \circledast \rho$ is the distribution on the right side of \eqref{finite sum of tails tail of sum}, hence in $RV_{-\alpha}$. From Theorem \ref{jacobsen} we conclude that $\nu$ and hence $X_{1}Y_{1} \in RV_{-\alpha}$. Finally we invoke Theorem~\ref{inverse product} to conclude that $X_{1} \in RV_{-\alpha}$.

\end{proof}

\section{Inverse problem for infinite sum} \label{sec:infinite sum}
We start with the same set-up as described in Section \ref{sec:finite sum}. We are now interested in sufficient conditions that ensure $X \in RV_{-\alpha}$ given that the infinite time ruin probability $\Psi(x) = P\left[\displaystyle \sup_{n \geq 1} S_{n} > x\right] \in RV_{-\alpha}$, where $S_{n}$ as in \eqref{discount value at n}. We shall state and prove two lemmas before the main result of this section. In this section, we shall additionally assume that $E\left[Y^{\alpha+\epsilon}\right] < 1$ for some $\epsilon \in (0, \alpha)$. This is required for finiteness of the geometric sum of the expectations. The next lemma shows the negligibility of the tail of the tail sums as well as the tail sum of the tails with respect to the dominator $W$.

\begin{lem} \label{lemma 1 infinite sum}
We start with the same set-up as in Lemma \ref{lemma 2 finite sum}, but additionally assume that $E[Y^{\alpha + \epsilon}] < 1$ for some $\epsilon \in (0, \alpha)$. Recall that $W \in RV_{-\alpha}$ is a bounded function with $P[XY > x] \leq W(x)$ for all $x > 0$. Then
\begin{equation} \label{eq1: lemma 1 infinite sum}
\lim_{m \rightarrow \infty} \limsup_{x \rightarrow \infty} \frac{P\left[\sum_{t=m+1}^{\infty} X_{t}\prod_{j=1}^{t}Y_{j} > x\right]}{W(x)} = 0,
\end{equation}
and
\begin{equation} \label{eq2: lemma 1 infinite sum}
\lim_{m \rightarrow \infty} \limsup_{x \rightarrow \infty} \frac{\sum_{t=m+1}^{\infty} P\left[X_{t}\prod_{j=1}^{t}Y_{j} > x\right]}{W(x)} = 0.
\end{equation}
\end{lem}

\begin{proof}
Using the notion of \emph{one large jump}, we split the numerator on the left side of \eqref{eq1: lemma 1 infinite sum} and bound it as follows.
\begin{equation} \label{split sum infinite lemma 1}
P\left[\sum_{t=m+1}^{\infty} X_{t}\prod_{j=1}^{t}Y_{j} > x\right] \leq \sum_{t=m+1}^{\infty}P\left[X_{t}\prod_{j=1}^{t}Y_{j} > x\right] + P\left[\sum_{t=m+1}^{\infty} X_{t}\prod_{j=1}^{t}Y_{j} \mathbf{1}_{\left[X_{t}\prod_{j=1}^{t}Y_{j} \leq x\right]} > x\right].
\end{equation}
If $G_{t}$ denotes the distribution of $\Theta_{t} = \prod_{j=1}^{t-1}Y_{j}$ then we have
\begin{equation} \label{bound first}
\sum_{t=m+1}^{\infty}\int_{0}^{\infty}W(x/u)G_{t}(du).
\end{equation}
For any $\gamma > \alpha$, from Karamata's theorem, we can find $M(\gamma) > 0$ such that
\begin{equation} \label{summand in second term}
E\left[\left\{X_{t}Y_{t}\right\}^{\gamma} \mathbf{1}_{[X_{t}Y_{t} \leq x]}\right] = \gamma \int_{0}^{x} u^{\gamma - 1} P[X_{t}Y_{t} > u] du \leq M(\gamma) W(x) x^{\gamma}.
\end{equation}
We now bound the second term on the right side of \eqref{split sum infinite lemma 1} separately for $\alpha < 1$ and $\alpha \geq 1$. For $\alpha < 1$, using Markov's inequality, and \eqref{summand in second term} for $\gamma = 1$, we have
\begin{align} \label{alpha less 1}
P\left[\sum_{t=m+1}^{\infty} X_{t}\prod_{j=1}^{t}Y_{j} \mathbf{1}_{\left[X_{t}\prod_{j=1}^{t}Y_{j} \leq x\right]} > x\right] \leq & \sum_{t=m+1}^{\infty} \int_{0}^{\infty} (x/v)^{-1} E\left[X_{t}Y_{t}\mathbf{1}_{[X_{t}Y_{t} \leq x/v]}\right] G_{t}(dv) \nonumber\\
\leq & M(1) \sum_{t=m+1}^{\infty} \int_{0}^{\infty} W(x/v) G_{t}(dv). 
\end{align}
For $\alpha \geq 1$, we use Markov's inequality, Minkowski's inequality, and \eqref{summand in second term} for $\gamma = \alpha + \epsilon$ to get the bound
\begin{equation} \label{alpha greater 1}
P\left[\sum_{t=m+1}^{\infty} X_{t}\prod_{j=1}^{t}Y_{j} \mathbf{1}_{\left[X_{t}\prod_{j=1}^{t}Y_{j} \leq x\right]} > x\right] \leq M(\alpha + \epsilon) \left\{\sum_{t=m+1}^{\infty} \left(\int_{0}^{\infty} W(x/v) G_{t}(dv)\right)^{\frac{1}{\alpha+\epsilon}}\right\}^{\alpha+\epsilon}.
\end{equation}
From \eqref{bound first}, \eqref{alpha less 1} and \eqref{alpha greater 1}, it suffices to show that
\begin{equation} \label{subcritical}
\lim_{m \rightarrow \infty} \limsup_{x \rightarrow \infty} \sum_{t=m+1}^{\infty} \int_{0}^{\infty} \frac{W(x/v)}{W(x)} G_{t}(dv) = 0, \quad \text{when } \alpha < 1,
\end{equation}
\begin{equation} \label{supercritical}
\lim_{m \rightarrow \infty} \limsup_{x \rightarrow \infty} \sum_{t=m+1}^{\infty} \left(\int_{0}^{\infty} \frac{W(x/v)}{W(x)} G_{t}(dv)\right)^{\frac{1}{\alpha+\epsilon}} = 0, \quad \text{when } \alpha \geq 1.
\end{equation}
We split the integral in \eqref{subcritical} over three intervals: $(0, 1], (1, x/x_{0}]$ and $(x/x_{0}, \infty)$, where $x_{0}$ is as in \eqref{Potter}. Then the integral over $(0, 1]$ is bounded by $M E\left[\Theta_{t}^{\alpha -\epsilon}\right]$, which is further bounded by $M \left\{E\left[Y^{\alpha + \epsilon}\right]^{\frac{\alpha - \epsilon}{\alpha + \epsilon}}\right\}^{t-1}$, using Potter's bounds (as in \eqref{Potter}) and Jensen's inequality. The integral over $(1, x/x_{0}]$ is bounded by $M \left\{E\left[Y^{\alpha + \epsilon}\right]\right\}^{t-1}$ again by Potter's bounds.
\par Because $W$ is bounded, the integral over $(x/x_{0}, \infty)$ is bounded as follows:
$$\int_{x/x_{0}}^{\infty}\frac{W(x/v)}{W(x)} G_{t}(dv) \leq \sup_{y > 0}W(y) \cdot \frac{P[\Theta_{t} > x/x_{0}]}{W(x)} \leq \sup_{y > 0}W(y) \cdot \frac{x_{0}^{\alpha+\epsilon} \left\{E\left[Y^{\alpha+\epsilon}\right]\right\}^{t-1}}{x^{\alpha+\epsilon} W(x)}.$$
Thus the final bound becomes, for a suitably large $M_{0}$,
\begin{equation} \label{final bound subcritical}
\int_{0}^{\infty} \frac{W(x/v)}{W(x)} G_{t}(dv) \leq M_{0}\left(\left\{E\left[Y^{\alpha + \epsilon}\right]\right\}^{\frac{(\alpha - \epsilon)(t-1)}{(\alpha + \epsilon)}} + \left\{E\left[Y^{\alpha + \epsilon}\right]\right\}^{t-1} + \frac{\left\{E\left[Y^{\alpha+\epsilon}\right]\right\}^{t-1}}{x^{\alpha+\epsilon} W(x)}\right).
\end{equation}
For $\alpha \geq 1$, from \eqref{final bound subcritical}, we get the bound 
\begin{equation} \label{final bound supercritical}
\left(\int_{0}^{\infty} \frac{W(x/v)}{W(x)} G_{t}(dv)\right)^{\frac{1}{\alpha+\epsilon}} \leq M_{0}^{\frac{1}{\alpha + \epsilon}} \left(\left\{E\left[Y^{\alpha + \epsilon}\right]\right\}^{\frac{(\alpha - \epsilon)(t-1)}{(\alpha + \epsilon)^{2}}} + \left\{E\left[Y^{\alpha + \epsilon}\right]\right\}^{\frac{t-1}{\alpha+\epsilon}} + \frac{\left\{E\left[Y^{\alpha+\epsilon}\right]\right\}^{\frac{t-1}{\alpha+\epsilon}}}{x W(x)^{\frac{1}{\alpha+\epsilon}}}\right).
\end{equation}
Because $W$ is bounded, the denominator $x W(x)^{\frac{1}{\alpha+\epsilon}} \rightarrow \infty$, and using the fact that $E[Y^{\alpha+\epsilon}] < 1$, we get the final desired results of \eqref{eq1: lemma 1 infinite sum} and \eqref{eq2: lemma 1 infinite sum}.
\end{proof}

\par We shall need one final lemma in order to show that, under the set-up described in Lemma \ref{lemma 1 infinite sum}, but now with $S = \sum_{i=1}^{\infty} X_{i} \prod_{j=1}^{i}Y_{j} \in RV_{-\alpha}$, the tail of $S$ is going to be asymptotically like the sum of the tails of the individual summands in $S$.

\begin{lem} \label{lemma 2 infinite sum}
Consider the exact same set-up as in Lemma \ref{lemma 1 infinite sum}, but now consider $W(x) = P[S > x]$ where $S = \sum_{i=1}^{\infty} X_{i} \prod_{j=1}^{i}Y_{j} \in RV_{-\alpha}$. Then 
\begin{equation} \label{tail of sum like sum of tail}
P[S > x] \sim \sum_{i=1}^{\infty} P\left[X_{i} \prod_{j=1}^{i} Y_{j} > x\right], \quad \text{as } x \rightarrow \infty.
\end{equation}
\end{lem}

\begin{proof}
From Lemma \ref{lemma 1 infinite sum}, for all $n \in \mathbb{N}$, we get
\begin{equation} \label{limit 1 infinite sum lemma 2}
\lim_{n \rightarrow \infty} \limsup_{x \rightarrow \infty} \frac{P\left[\sum_{t=n+1}^{\infty} X_{t} \prod_{j=1}^{t}Y_{j} > x\right]}{P[S > x]} = 0,
\end{equation}
\begin{equation} \label{limit 2 infinite sum lemma 2}
\lim_{n \rightarrow \infty} \limsup_{x \rightarrow \infty} \sum_{t=n+1}^{\infty} \frac{P\left[X_{t} \prod_{j=1}^{t}Y_{j} > x\right]}{P[S > x]} = 0,
\end{equation}
and from Lemma \ref{lemma 2 finite sum}, for all $s \neq t$, we have
\begin{equation} \label{limit 3 infinite sum lemma 2}
\lim_{x \rightarrow 0} \frac{P\left[X_{s} \prod_{j=1}^{s}Y_{j} > x, \ X_{t} \prod_{j=1}^{t}Y_{j} > x\right]}{P[S > x]} = 0.
\end{equation}
Recall $S_{n}$ as defined in \eqref{discount value at n}. For any $\delta > 0$, 
\begin{equation} \label{split sum lemma 2 infinite sum}
P[S > (1+\delta)x] \leq P[S_{n} > x] + P\left[\sum_{t=n+1}^{\infty} X_{t} \prod_{j=1}^{t} Y_{j} > \delta x\right].
\end{equation}
From \eqref{limit 1 infinite sum lemma 2} and because $S \in RV_{-\alpha}$, we get
 
$$\lim_{n \rightarrow \infty} \liminf_{x \rightarrow \infty}\frac{P[S_{n} > x]}{P[S > x]} \geq \liminf_{x \rightarrow \infty} \frac{P[S > (1+\delta)x]}{P[S > x]} - \lim_{n \rightarrow \infty} \limsup_{x \rightarrow \infty}\frac{P\left[\sum_{t=n+1}^{\infty} X_{t} \prod_{j=1}^{t} Y_{j} > \delta x\right]}{P[S > x]} = (1+\delta)^{-\alpha},$$
so that by letting $\delta \rightarrow 0$, we get the lower bound 
\begin{equation} \label{liminf lower bound}
\lim_{n \rightarrow \infty} \liminf_{x \rightarrow \infty}\frac{P[S_{n} > x]}{P[S > x]} \geq 1.
\end{equation}
But we trivially also have $P[S_{n} > x] \leq P[S > x]$, hence we conclude
\begin{equation} \label{limsup upper bound}
\lim_{n \rightarrow \infty} \limsup_{x \rightarrow \infty}\frac{P[S_{n} > x]}{P[S > x]} \leq 1.
\end{equation}
We invoke the inequalities from Lemma 4.3 of \cite{hazra:maulik:2012}, as in Lemma \ref{inverse finite sum}, and consider \eqref{lower bound} and \eqref{upper bound}. From \eqref{lower bound}, \eqref{limsup upper bound} and \eqref{limit 3 infinite sum lemma 2}, we have
$$\lim_{n \rightarrow \infty} \limsup_{x \rightarrow \infty} \frac{\sum_{t=1}^{n} P\left[X_{t}\prod_{j=1}^{t}Y_{j} > x\right]}{P[S > x]} \leq 1.$$
Finally, from \eqref{upper bound}, \eqref{liminf lower bound} and again \eqref{limit 3 infinite sum lemma 2}, and using the regular variation of $S$, we get
$$\lim_{n \rightarrow \infty} \liminf_{x \rightarrow \infty} \frac{\sum_{t=1}^{n} P\left[X_{t}\prod_{j=1}^{t}Y_{j} > x\right]}{P[S > x]} \geq 1.$$
Combining these with \eqref{limit 2 infinite sum lemma 2} we get the final result.
\end{proof}

We finally come to the main result of this section, which infers about the tail behavior of each $X_{i}$ from the regularly varying tail of $S = \sum_{i=1}^{\infty} X_{i} \prod_{j=1}^{i} Y_{j}$. 

\begin{thm} \label{inverse infinite sum}
Consider the exact same set-up as in Lemma \ref{lemma 2 infinite sum}. Additionally, we assume that for all $\beta \in \mathbb{R}$, \eqref{product Mellin} holds and 
\begin{equation} \label{infinite sum Mellin}
\sum_{k=0}^{\infty} \left\{E\left[Y^{\alpha + i \beta}\right]\right\}^{k} \neq 0.
\end{equation}
Then we conclude that each $X_{i} \in RV_{-\alpha}$ and 
\begin{equation} \label{eq: inverse infinite sum}
P[S > x] \sim \frac{E[Y^{\alpha}] + \theta d_{1} E[\phi_{2}(Y) Y^{\alpha}]}{1-E[Y^{\alpha}]} \overline{F}(x) \quad \text{as } x \rightarrow \infty.
\end{equation}
\end{thm}

\begin{proof}
The proof is similar to that of Theorem \ref{inverse finite sum} with $\rho$ defined as
$$\rho(B) = \sum_{k=1}^{\infty} P\left[\prod_{j=1}^{k-1}Y_{j} \in B\right], \ B \text{ a Borel set in } (0, \infty).$$
%From the moment condition $E[Y^{\alpha + \epsilon}] < 1$, we conclude that $\rho$ is $\sigma$-finite. Observe that $\nu \circledast \rho$ is the %distribution 
%$$\sum_{i=1}^{\infty} P\left[X_{i} \prod_{j=1}^{i} Y_{j} > x\right],$$ 
%and from Lemma \ref{lemma 2 infinite sum} and because $S \in RV_{-\alpha}$, we get $\nu \circledast \rho \in RV_{-\alpha}$ as well. The moment %condition \eqref{moment condition} holds because of $E[Y^{\alpha + \epsilon}] < 1$. The non-vanishing Mellin transform condition of \eqref{non-vanishing %Mellin transform} holds because of \eqref{infinite sum Mellin}. Thus applying Theorem \ref{jacobsen} we have each $X_{i}Y_{i} \in RV_{-\alpha}$ and $P[S %> x] \sim \frac{1}{1-E[Y^{\alpha}]} P[X_{i}Y_{i} > x]$. Since \eqref{product Mellin} holds, we have, applying Theorem \ref{inverse product}, that each %$X_{i} \in RV_{-\alpha}$, and \eqref{eq: inverse infinite sum} holds.
\end{proof}

\section{Necessity of the non-vanishing Mellin transform condition} \label{non-vanishing Mellin}
Each of Theorems \ref{inverse product}, \ref{inverse finite sum} and \ref{inverse infinite sum} has non-vanishing Mellin transform condition(s) imposed on the sequence $\{Y_{i}\}$ of random variables, similar to the condition \eqref{non-vanishing Mellin transform} in Theorem \ref{jacobsen}. In this section we shall show that such a condition cannot be relaxed for proving our results. This is similar to the assertion made by \cite{jacobsen:2009} in Theorem 2.3. They show that if \eqref{non-vanishing Mellin transform} does not hold for some $\beta$, then a $\sigma$-finite measure $\nu$ without a regularly varying tail can be found such that $\nu \circledast \rho$ is regularly varying. The construction of the counterexample in Theorem \ref{counterexample} is inspired by Jacobsen et al. 
\par To this end, recall the class of dominatedly tail varying distributions given in \eqref{dominated tail}. We start with the following useful remark.
\begin{rem} \label{reg var bounded}
From \cite{foss:korshunov:zachary}, we know that, if $F$ and $G$ are distribution functions with $G \in RV_{-\alpha}$ and for some constants $0 < c_{1} < c_{2} < \infty$,
$$c_{1}\overline{G}(x) \leq \overline{F}(x) \leq c_{2}\overline{G}(x) \quad \text{for all sufficiently large } x,$$
then $F \in \mathcal{D}$.
\end{rem} 

\begin{thm} \label{counterexample}
Let $G$ be a distribution function on $(0, \infty)$. We are given two bounded functions $\phi_{1}$ and $\phi_{2}$ on $(0, \infty)$, and $\theta \in \mathbb{R}$ such that:
\begin{enumerate}
\item $\phi_{1}$ takes both positive and negative values,
\item $\lim_{x \rightarrow \infty} \phi_{1}(x) = d_{1} \in \mathbb{R}$ exists,
\item for all $x > 0, y > 0$ we have $1 + \theta \phi_{1}(x) \phi_{2}(y) \geq 0$,
\item $\int_{0}^{\infty} \phi_{2}(y) G(dy) = 0$.
\end{enumerate}
For some $\alpha > 0, \epsilon > 0$ and $\beta_{0} \in \mathbb{R}$, assume that $\int_{0}^{\infty} y^{\alpha + \epsilon} G(dy) < \infty$ and 
\begin{equation} \label{vanishing Mellin}
\int_{0}^{\infty} y^{\alpha + i \beta_{0}} G(dy) + \theta d_{1} \int_{0}^{\infty} \phi_{2}(y) y^{\alpha + i \beta_{0}} G(dy) = 0.
\end{equation}
If $Y \sim G$, then there exists $X$, with not regularly varying tail, such that $(X, Y)$ is jointly distributed as bivariate Sarmanov, as defined in Definition \ref{sarmanov definition}, with kernel functions $\phi_{1}, \phi_{2}$, and constant $\theta$, with $XY$ having regularly varying tail with index $-\alpha$.
\end{thm}

\begin{proof}
Find $Y$ such that its marginal is $G$, then consider its twisted version $Y_{\theta}^{*}$ as defined in \eqref{twisted version}, with marginal $G_{\theta}$ given by $G_{\theta}(dy) = (1 + \theta d_{1} \phi_{2}(y)) G(dy)$. From the condition of \eqref{vanishing Mellin}, we find that $$E[{Y_{\theta}^{*}}^{\alpha + i \beta_{0}}] = 0.$$ Since $\phi_{2}$ is bounded, let $|\phi_{2}(y)| \leq b_{2}$ for all $y > 0$, for some finite $b_{2}$. Then
\begin{equation} \label{moment condition counterexample}
E[{Y_{\theta}^{*}}^{\alpha + \epsilon}] \leq (1 + |\theta| |d_{1}| b_{2}) \int_{0}^{\infty} y^{\alpha + \epsilon} G(dy) < \infty.
\end{equation}
This shows that the moment condition of \eqref{moment condition}, $$\int_{0}^{\infty} y^{\alpha - \epsilon} \vee y^{\alpha + \epsilon} G_{\theta}(dy) < \infty,$$ holds. We adopt the idea of Theorem 2.1 of \cite{jacobsen:2009} to define two distribution functions $F_{\alpha}$ and $\widetilde{F}$ as follows:
$$F_{\alpha}(dx) = \alpha x^{-(\alpha + 1)}dx, \quad \text{for all } x > 1,$$
which means $F \in RV_{-\alpha}$, and 
$$\widetilde{F}(dx) = g(x) F_{\alpha}(dx), \quad x > 1,$$
where $g(x) = 1 + a \cos \left(\beta_{0} \log x\right) + b \sin \left(\beta_{0} \log x\right)$ for some constants $a > 0, b > 0$ with $0 < a + b \leq 1$. Then again, from the conclusion of Theorem 2.1 of \cite{jacobsen:2009}, we have $\overline{\widetilde{F} \circledast G_{\theta}} \sim \overline{F_{\alpha} \circledast G_{\theta}} \in RV_{-\alpha}$. But from Theorem 2.3 of \cite{jacobsen:2009}, we know that $\widetilde{F}$ does not have a regularly varying tail.
\par We now have to tweak $\widetilde{F}$ to get our desired $F$, so that all conditions of Theorem \ref{counterexample} hold. Choosing $c > 1$ so that $\overline{\widetilde{F}}(c) < 1$, we define a new distribution $F^{(1)}$ as follows:
\[\overline{F^{(1)}}(y)=\left\{\begin{array}{ll}
\overline{\widetilde{F}}(y) & \mbox{for $y > c$,}\\
\overline{\widetilde{F}}(c) & \mbox{for $1 < y \leq c$,}\\
1 &\mbox{for $y \leq 1$.}
\end{array}\right.\]
We shall now show that $\overline{F^{(1)} \circledast G_{\theta}} \sim \overline{\widetilde{F} \circledast G_{\theta}}$. Observe that
\begin{align} \label{split sum counterexample}
\overline{F^{(1)} \circledast G_{\theta}}(x) =& \int_{(0, x/c)} + \int_{[x/c, x)} + \int_{[x, \infty)} \overline{F^{(1)}}\left(\frac{x}{u}\right) G_{\theta}(du) \nonumber\\
=& \overline{\widetilde{F} \circledast G_{\theta}}(x) - \int_{x/c}^{\infty} \overline{\widetilde{F}}\left(\frac{x}{u}\right) G_{\theta}(du)  + \overline{\widetilde{F}}(c) G_{\theta}([x/c, x)) + G_{\theta}([x, \infty)). 
\end{align} 
We deal with the second term of the sum in \eqref{split sum counterexample} first. Let $\|G_{\theta}\|_{\alpha}$ denote the intergal $\int_{0}^{\infty} y^{\alpha} G_{\theta}(dy)$. From the definitions of $F_{\alpha}$ and $\widetilde{F}$, and the fact that $\overline{\widetilde{F} \circledast G_{\theta}} \sim \overline{F_{\alpha} \circledast G_{\theta}}$, we get:
\begin{align} \label{counterexample second term}
\lim_{x \rightarrow \infty} \frac{\int_{x/c}^{\infty} \overline{\widetilde{F}}\left(\frac{x}{u}\right) G_{\theta}(du)}{\overline{\widetilde{F} \circledast G_{\theta}}(x)} = & \lim_{x \rightarrow \infty} \frac{\int_{x/c}^{\infty} \overline{\widetilde{F}}\left(\frac{x}{u}\right) G_{\theta}(du)}{\overline{F_{\alpha} \circledast G_{\theta}}(x)} \nonumber\\
=& \lim_{x \rightarrow \infty} \frac{\int_{x/c}^{\infty} \int_{x/u}^{\infty} g(y) F_{\alpha}(dy) G_{\theta}(du)}{x^{-\alpha} \|G_{\theta}\|_{\alpha}} \nonumber\\
\leq & \lim_{x \rightarrow \infty} \frac{ (1 + a + b) \int_{x/c}^{\infty} \overline{F_{\alpha}}\left(\frac{x}{u}\right) G_{\theta}(du)}{x^{-\alpha} \|G_{\theta}\|_{\alpha}} \nonumber\\
%=& \lim_{x \rightarrow \infty} \frac{(1 + a + b)x^{-\alpha} \int_{x/c}^{\infty} u^{\alpha} G_{\theta}(du)}{x^{-\alpha} \|G_{\theta}\|_{\alpha}} %\nonumber\\
=& \lim_{x \rightarrow \infty} \frac{(1 + a + b) \int_{x/c}^{\infty} u^{\alpha} G_{\theta}(du)}{\|G_{\theta}\|_{\alpha}}. 
\end{align}
From \eqref{moment condition counterexample} we conclude that $\int_{0}^{\infty} y^{\alpha} G_{\theta}(dy) < \infty$, so that applying dominated convergence, the numerator on the right side of \eqref{counterexample second term} goes to $0$ as $x \rightarrow \infty$. We now consider the last two summands on the right side of \eqref{split sum counterexample}.

\begin{align}
\lim_{x \rightarrow \infty} \frac{\overline{\widetilde{F}}(c) G_{\theta}([x/c, x)) + G_{\theta}([x, \infty))}{\overline{\widetilde{F} \circledast G_{\theta}}(x)} \leq & \lim_{x \rightarrow \infty} \frac{\overline{\widetilde{F}}(c) \overline{G_{\theta}}(x/c) + \overline{G_{\theta}}(x)}{\overline{F_{\alpha} \circledast G_{\theta}}(x)} \nonumber\\
\leq & \lim_{x \rightarrow \infty} \frac{\overline{\widetilde{F}}(c) \frac{\int_{0}^{\infty} u^{\alpha + \epsilon} G_{\theta}(du)}{(x/c)^{\alpha + \epsilon}} + \frac{\int_{0}^{\infty} u^{\alpha + \epsilon} G_{\theta}(du)}{x^{\alpha + \epsilon}}}{x^{-\alpha} \|G_{\theta}\|_{\alpha}} \nonumber\\
=& \lim_{x \rightarrow \infty} \frac{\left(1 + c^{\alpha + \epsilon} \overline{\widetilde{F}}(c)\right) \int_{0}^{\infty} u^{\alpha + \epsilon} G_{\theta}(du)}{x^{\epsilon} \|G_{\theta}\|_{\alpha}} = 0, \nonumber
\end{align}
due to \eqref{moment condition counterexample}. 
%Ultimately, we have $F^{(1)} \circledast G_{\theta} = \widetilde{F} \circledast G_{\theta}$. Consequently $F^{(1)} \circledast G_{\theta} \in RV_{-\alpha}$. 
\par From the definitions of $F^{(1)}$ and $\widetilde{F}$ in terms of $F_{\alpha}$, and Remark \ref{reg var bounded}, it is immediate that $F^{(1)} \in \mathcal{D}$. As $\widetilde{F}$ and $F^{(1)}$ eventually have the same tail, $F^{(1)}$ cannot be regularly varying. The last step in this proof is to adjust $F^{(1)}$ slightly to get the final desired distribution $F$ so that $\int_{0}^{\infty} \phi_{1}(x) F(dx) = 0$.
\par For this purpose, we define $\hat{\phi_{1}}$ as $\hat{\phi_{1}}(x) = \int_{x}^{\infty} \phi_{1}(x) F^{(1)}(dx), \ x > 0$. Because $\phi_{1}$ is bounded, $\hat{\phi_{1}}$ is continuous on $(1, \infty)$. We now subdivide into three cases:
\begin{enumerate}
\item \label{1} If $\hat{\phi_{1}}$ takes both positive and negative values on $(1, \infty)$, by intermediate value theorem, we find $x_{0} > 1$ such that $\hat{\phi_{1}}(x_{0}) = \int_{x_{0}}^{\infty} \phi_{1}(x) F^{(1)}(dx) = 0$. Then we define $F$ as $\overline{F}(x) = \overline{F^{(1)}}(x)/\overline{F^{(1)}}(x_{0})$ for $x \geq x_{0}$.

\item \label{2} Suppose $\hat{\phi_{1}}$ takes only strictly positive values on $(0, \infty)$. Because $\phi_{1}$ takes both positive and negative values, we find $x_{1} > 0$ and $c_{1} > 0$ such that $\phi_{1}(x_{1}) = -c_{1}$. Let $\hat{\phi_{1}}(1) = \int_{1}^{\infty} \phi_{1}(x) F^{(1)}(dx) = c_{0} > 0$. We define the probability measure $\mu$ as follows:
$$\mu(B) = \frac{\mu^{(1)}\left[B \bigcap (1, \infty)\right] + \frac{c_{0}}{c_{1}} \delta_{x_{1}}(B)}{\mu^{(1)}(1, \infty) + \frac{c_{0}}{c_{1}}}, \quad B \text{ a Borel set on } (0, \infty),$$
where $\mu^{(1)}$ is the law induced by $F^{(1)}$. Then we take $F$ to be the distribution function for $\mu$.

\item \label{3} Suppose $\hat{\phi_{1}}$ takes only strictly negative values on $(0, \infty)$. Again, we can find $x_{2} > 0$ and $c_{2} > 0$ such that $\phi_{1}(x_{1}) = c_{2}$. Let $\hat{\phi_{1}}(1) = \int_{1}^{\infty} \phi_{1}(x) F^{(1)}(dx) = -c_{0} < 0$. We define the probability measure $\mu$ as follows:
$$\mu(B) = \frac{\mu^{(1)}\left[B \bigcap (1, \infty)\right] + \frac{c_{0}}{c_{2}} \delta_{x_{2}}(B)}{\mu^{(1)}(1, \infty) + \frac{c_{0}}{c_{2}}}, \quad B \text{ a Borel set on } (0, \infty).$$
Then we take $F$ to be the distribution function for $\mu$.
\end{enumerate}

\par We claim that for a suitable constant $\kappa$, $F \circledast G_{\theta} \sim \kappa F^{(1)} \circledast G_{\theta}$, which gives $F \circledast G_{\theta} \in RV_{-\alpha}$. This is immediate for (\ref{1}). For (\ref{2}), we consider
\begin{align} \label{last step split sum}
\overline{F \circledast G_{\theta}}(x) =& \frac{1}{\mu^{(1)}(1, \infty) + \frac{c_{0}}{c_{1}}} \left[\int_{0}^{x} \overline{F^{(1)}}\left(\frac{x}{u}\right) G_{\theta}(du) + \overline{F^{(1)}}(1) G_{\theta}([x, \infty)) + \frac{c_{0}}{c_{1}} G_{\theta}\left(\left[\frac{x}{x_{1}}, \infty\right)\right)\right] \nonumber\\
=& \frac{1}{\overline{F^{(1)}}(1) + \frac{c_{0}}{c_{1}}} \left[\overline{F^{(1)} \circledast G_{\theta}}(x) - \int_{x}^{\infty} \overline{F^{(1)}}\left(\frac{x}{u}\right) G_{\theta}(du) + \overline{F^{(1)}}(1) G_{\theta}([x, \infty)) + \frac{c_{0}}{c_{1}} G_{\theta}\left(\left[\frac{x}{x_{1}}, \infty\right)\right)\right].
\end{align}
We now deal with the second term in the sum on the right side of \eqref{last step split sum} the same way as the second term in the sum of \eqref{split sum counterexample}, and the sum of the last two terms in the same way as the last two terms of \eqref{split sum counterexample}. From the definition of $F$ in terms of $F^{(1)}$ and hence $\widetilde{F}$, and because $\widetilde{F}$ is not regularly varying, we conclude that $F$ also not regularly varying. Case (\ref{3}) is dealt with similarly. This completes the proof. 
\end{proof}

\section{Acknowledgement}
The first author kindly acknowledges partial support by the project RARE-318984, a Marie Curie IRSES Fellowship within the 7th European Community Framework Programme. The work forms part of the Dissertation for the Master of Statistics degree of the second author at Indian Statistical Institute, Kolkata, India.

\bibliography{mybibfile2}

\begin{thebibliography}{16}
\providecommand{\natexlab}[1]{#1}
\providecommand{\url}[1]{\texttt{#1}}
\expandafter\ifx\csname urlstyle\endcsname\relax
  \providecommand{\doi}[1]{doi: #1}\else
  \providecommand{\doi}{doi: \begingroup \urlstyle{rm}\Url}\fi

\bibitem[Basrak et~al.(2002)Basrak, Davis, and Mikosch]{basrak:2002}
B.~Basrak, R.~A. Davis, and T.~Mikosch.
\newblock A characterization of multivariate regular variation.
\newblock \emph{Annals of Applied Probability}, 12:\penalty0 908--920, 2002.

\bibitem[Bingham and Teugels(1987)]{goldie:1987}
Goldie C.~M. Bingham, N.~H. and J.~L. Teugels.
\newblock Regular variation.
\newblock \emph{Encyclopedia of Mathematics and its Applications}, 27, 1987.

\bibitem[Breiman(1965)]{breiman:1965}
Leonard Breiman.
\newblock On some limit theorems similar to the arc-sin law.
\newblock \emph{Theory of Probability \& Its Applications}, 10\penalty0
  (2):\penalty0 323--331, 1965.

\bibitem[Cline and Samorodnitsky(1994)]{cline:1994}
D.~B.~H. Cline and G.~Samorodnitsky.
\newblock Subexponentiality of the product of independent random variables.
\newblock \emph{Stochastic Processes and their Applications}, 49:\penalty0
  75--98, 1994.

\bibitem[Damek et~al.(2014)Damek, Mikosch, Rosi\'{n}ski, and
  Samorodnitsky]{damek:2014}
E.~Damek, T.~Mikosch, J.~Rosi\'{n}ski, and G.~Samorodnitsky.
\newblock General inverse problems for regular variation.
\newblock \emph{Journal of Applied Probability}, 51\penalty0 (A):\penalty0
  229--248, 2014.

\bibitem[Denisov and Zwart(2007)]{denisov:zwart:2007}
Denis Denisov and Bert Zwart.
\newblock On a theorem of breiman and a class of random difference equations.
\newblock \emph{Journal of Applied Probability}, 44\penalty0 (4):\penalty0
  1031--1046, 2007.

\bibitem[Foss et~al.(2013)Foss, Korshunov, and Zachary]{foss:korshunov:zachary}
Sergey Foss, Dmitry Korshunov, and Stan Zachary.
\newblock \emph{An introduction to heavy-tailed and subexponential
  distributions}.
\newblock Springer, 2013.

\bibitem[Hazra and Maulik(2012)]{hazra:maulik:2012}
Rajat~Subhra Hazra and Krishanu Maulik.
\newblock Tail behavior of randomly weighted sums.
\newblock \emph{Advances in Applied Probability}, 44\penalty0 (3):\penalty0
  794--814, 2012.

\bibitem[Jacobsen et~al.(2009)Jacobsen, Mikosch, Rosi\'{n}ski, and
  Samorodnitsky]{jacobsen:2009}
M.~Jacobsen, T.~Mikosch, J.~Rosi\'{n}ski, and G.~Samorodnitsky.
\newblock Inverse problems for regular variation of linear filters, a
  cancellation property for $\sigma$-fiinite measures and identification of
  stable laws.
\newblock \emph{Annals of Applied Probability}, 19\penalty0 (1):\penalty0
  210--242, 2009.

\bibitem[Lee(1996)]{lee:1996}
M.~L.~T. Lee.
\newblock Properties and applications of the sarmanov family of bivariate
  distributions.
\newblock \emph{Communications in Statistics. -Theory Meth.}, 25\penalty0
  (6):\penalty0 1207--1222, 1996.

\bibitem[Maulik and Podder()]{maulik:podder:2016}
Krishanu Maulik and Moumanti Podder.
\newblock Ruin probabilities under sarmanov dependence structure.
\newblock Submitted, Pre-print \url{http://arxiv.org/abs/1601.04637},.

\bibitem[Nyrhinen(2012)]{nyrhinen:2012}
Harri Nyrhinen.
\newblock On stochastic difference equations in insurance ruin theory.
\newblock \emph{Journal of difference equations and applications}, 18\penalty0
  (8):\penalty0 1345--1353, 2012.

\bibitem[Paulsen(2008)]{paulsen:2008}
Jostein Paulsen.
\newblock Ruin models with investment income.
\newblock \emph{Probability Surveys}, 5:\penalty0 416–434, 2008.

\bibitem[Resnick(1987)]{resnick:1987}
S.~I. Resnick.
\newblock \emph{Extreme values, Regular variation and Point processes}.
\newblock Springer, New York, 1987.

\bibitem[Resnick and Willekens(1991)]{resnick:willekens:1991}
Sidney~I Resnick and Eric Willekens.
\newblock Moving averages with random coefficients and random coefficient
  autoregressive models.
\newblock \emph{Stochastic Models}, 7\penalty0 (4):\penalty0 511--525, 1991.

\bibitem[Yang and Wang(2013)]{yang:wang:2013}
Yang Yang and Yuebao Wang.
\newblock Tail behavior of the product of two dependent random variables with
  applications to risk theory.
\newblock \emph{Extremes}, 16\penalty0 (1):\penalty0 55--74, 2013.

\end{thebibliography}
\end{document}